\documentclass{amsart}  
\usepackage[table]{xcolor}

\usepackage{styg0}
\usepackage{graphics}
\usepackage{bm}

\usepackage{hyperref}
\hypersetup{colorlinks,citecolor=blue,linktocpage,hyperindex=true,backref=true}

\usepackage{array,longtable,float}

\usepackage{calc,tikz}
\usepackage{xcolor}
\usepackage{array,booktabs,colortbl}

\usepackage{todonotes}
\date{June 10, 2023}

\title[Reflective $2$-elementary lattices]{Reflective hyperbolic $2$-elementary
  lattices, K3 surfaces and hyperkahler manifolds}
\author{Valery Alexeev}
\email{valery@uga.edu}
\address{Department of Mathematics, University of Georgia, Athens GA
  30602, USA}
\subjclass[2010]{14J28, 20F55}

\begin{document} 

\begin{abstract}
  We compute Coxeter diagrams of several ``large'' reflective
  $2$-elementary even hyperbolic lattices and their maximal parabolic
  subdiagrams, and give some applications of these results to the
  theory of K3 surfaces and hyperkahler manifolds.
\end{abstract}

\maketitle 
%{\centering\footnotesize To Vyacheslav Vladimirovich Shokurov on his 70th birthday\par} 

%\setcounter{tocdepth}{1}
\setcounter{tocdepth}{2} 
\tableofcontents

\section{Introduction}\label{sec:intro}

Even hyperbolic lattices $L$ appear in the theory of K3 surfaces and
hyperkahler manifolds in several ways:

\smallskip (1) As the Picard lattice $S_X$ of a K3 surface $X$
together with the intersection form, resp. as the Picard lattice
$S_X\subset H^2(X,\bZ)$ of a hyperkahler
manifold $X$, with the Beauville-Bogomolov-Fujiki form.  In
this interpretation, the primitive isotropic vectors $v\in S_X$,
$v^2=0$, correspond to elliptic, resp. Lagrangian fibrations.

(2) As the hyperbolic lattice $\delta^\perp/\delta$ attached to a
$0$-cusp of the moduli space of lattice polarized K3 surfaces,
resp. hyperkahler manifolds.  In this interpretation, the
isotropic vectors correspond to $1$-cusps of the moduli space and to
Type II degenerations. (Here, $\delta$ is an isotropic vector in some
lattice $H$ of signature $(2,n)$ for some $n$, defining a $0$-cusp;
then $\delta^\perp/\delta$ has signature $(1,n-1)$.)

\smallskip Let $L$ be such a lattice. Since we are mostly concerned
with application to algebraic geometry, ``hyperbolic'' for us means
signature $(1,r-1)$.  The isometry group $O(L)$ contains two important
normal subgroups $W_2(L)$ and $W_r(L)$ generated by reflections in the
$(-2)$-vectors, resp. in all reflective vectors $r\in L$ with $r^2<0$.
They are described by the Coxeter diagrams $\Gamma_2$,
resp. $\Gamma_r$.  The lattice is called $2$-reflective,
resp. reflective if $W_2(L)$, resp. $W_r(L)$ has finite index in
$O(L)$.

By the Torelli theorem, the automorphism group $\Aut(X)$ of a K3
surface is commensurable to $O(S_X)/W_2(S_X)$, in particular $\Aut(X)$
is finite iff $S_X$ is $2$-reflective.  In some cases when $S_X$ is
not $2$-reflective but is reflective, the Coxeter diagram $\Gamma_r$
provides a nice description of the infinite group $\Aut(X)$, see
Vinberg \cite{vinberg1983two-most} for a classical example.  In the
case of hyperkahler manifolds, similar results hold for the group
$\Bir(X)$ of birational automorphisms.  In the second interpretation
above, reflective lattices provide toroidal compactifications of the
moduli spaces.

A simple but very powerful way to compute the Coxeter diagrams
$\Gamma_2$ and $\Gamma_r$ is Vinberg's algorithm
\cite{vinberg1972-groups-of-units, vinberg1973some-arithmetic}.  When
the lattice is $2$-reflective, resp. reflective, the algorithm
completes in finitely many steps. Then the $O(L)$-orbits of primitive
isotropic vectors correspond to the maximal parabolic subdiagrams.

\smallskip

In this paper we treat the case of $2$-elementary lattices,
i.e. the lattices with $L^*/L\simeq \bZ_2^a$.  By Nikulin
\cite{nikulin1979integer-symmetric} an indefinite even $2$-elementary
lattice is uniquely determined by its signature and a triple of
integers $(r,a,\delta)$, where $r$ is its rank, $a\equiv r\pmod 2$ is
the $\bZ_2$-rank, and $\delta\in\{0,1\}$ is a certain invariant which
we call coparity.  In the case of K3 surfaces, these are the generic
Picard lattices of K3 surfaces with a nonsymplectic involution.

In some sense $2$-elementary lattices are one step up from unimodular
lattices. For such lattices $\Gamma_2$ is easily derived from $\Gamma_r$,
even if one or both of these diagrams are infinite, 
cf.~\cite[Prop.~on p.2]{vinberg1983two-most}
or~\cite[Prop.~2.4]{alexeev2006del-pezzo}.  In any case, $\Gamma_r$ is
usually dramatically smaller than $\Gamma_2$, so it is a good strategy to
compute $\Gamma_r$.

To begin with, there are only finitely many reflective $2$-elementary
lattices. Indeed, $a\le r$, and by Esselmann
\cite{esselmann1996uber-die-maximale} a reflective lattice satisfies
$r\le 20$ or $r=22$.  For lattices with $r+a\le 18$ the Coxeter
diagrams $\Gamma_r$ were computed by Nikulin
in~\cite[Table~1]{alexeev2006del-pezzo}. For $r+a=20$ we computed them
in \cite{alexeev2022mirror-symmetric}. The latter paper also contains
a description of the isotropic vectors for $r+a\le 20$. 

For K3 surfaces over $\bC$ the Picard lattice satisfies $r+a\le
22$. The initial motivation for this paper was to classify the
reflective diagrams $\Gamma_r$ with $r+a=22$ and to compute the
maximal parabolic subdiagrams in them. 

The diagrams for $(11,11,1)$, $(12,10,1)$, $(13,9,1)$, $(14,8,0)$ and
$(14,8,1)$ are relatively small and we give them in
Figure~\ref{fig:coxeter}. The others are too large to draw
directly. However, it turns out that they are remarkably symmetric and
can be described in a way similar to the
papers~\cite{vinberg1978the-groups,vinberg1983two-most} of Vinberg and
Kaplinskaja.  

\begin{theorem}\label{thm1}
  The $2$-elementary lattices $(17,5,1)$, $(18,4,1)$ and $(19,3,1)$
  are not reflective. The other even $2$-elementary lattices on the line
  $r+a=22$ are reflective.
\end{theorem}

\begin{definition}\label{def:nth-graph}
  Let $G$ be a graph and $n\ge2$ be an integer. Define the edge
  $n$-fold graph $G^{(n)}$ by subdividing each edge of $G$ into $n$
  edges and inserting $n-1$ intermediate vertices. Thus, $G^{(n)}$ has
  $n|E_G|$ edges and $|V_G| + (n-1)|E_G|$ vertices.
\end{definition}

We note that the vertices of $G^{(2)}$ are in a natural bijection with
the set of vertices and edges of $G$.

\begin{definition}\label{def:common-graph}
  We say that a diagram $\Gamma$ \emph{is built on top of a simple
    graph $H$} if it contains a subdiagram of the \emph{main roots}
  isomorphic to $H$, and the \emph{additional roots} are defined in
  terms of the main roots by some specified rules.  For the
  $2$-elementary lattices of this paper the main roots are the
  $(-2)$-roots $\alpha$ in $L$ of divisibility~$1$, i.e. satisfying
  $\alpha\cdot L=\bZ$.
  For the lattices in Theorem~\ref{thm2} the
  following is also true: if a main root $\alpha$ and an additional
  root $\beta$ are connected then $\alpha\cdot\beta=2$.
\end{definition}

\begin{theorem}\label{thm2}
  For $n=3,4,5,6$ the Coxeter diagram $\Gamma_r$ of the even
  $2$-elementary lattice $(10+n,12-n,1)$ is built on top of 
  $G^{(2)}$, where $G=K_n$ is the complete graph on $n$ vertices. The
  diagram $\Gamma_r$ for $(18,4,0)$ is built on top of $G^{(2)}$,
  where $G=K_{4,4}$ is the complete bipartite graph on $8$ vertices,
  and that of $(14,8,0)$ on top of $G^{(2)}$ for $G=D_4$.  In all
  cases one has
  $\Aut\Gamma_r = \Aut G$ and $O^+(S) = \Aut G \ltimes W_r$.
\end{theorem}

In some cases, e.g. $(11,11,1)$, $(18,4,0)$, $(20,2,1)$, the diagrams
were previously known.  We add the computations of the maximal
parabolic diagrams, which are not easy to extract from the literature.

\smallskip

We also investigated the remaining finitely many possibly reflective
$2$-elementary lattices. Some of these lattices appear as the Picard lattices
of hyperkahler manifolds of
$K3^{[2]}$-type. Indeed if $X$ is a K3 surface then
$S(\Hilb^n X) = \la 2-2n\ra \oplus S_X$. In particular, if $S_X$ is
$2$-elementary then $S(\Hilb^2 X) = \la -2\ra \oplus S_X$ is also
$2$-elementary. The lattice $(15,9,1)=\la -2\ra\oplus (14,8,1)$ below
provides such an example 
with a very interesting automorphism group,
possibly related to the Klein quartic.

There are similar identities for the hyperkahler manifolds of Kummer
and OG6 types, but those lattices are quite small. (In the 
OG10 type the Picard lattices are of the form
$A_2\oplus S_X$ and so are not $2$-elementary since $A_2^*/A_2=\bZ_3$.)

And some of the bigger lattices appear on K3 surfaces and hyperkahler
manifolds in positive characteristic, for example the lattice
$(22,2,0)$ investigated by Dolgachev and Kond\=o
\cite{dolgachev2003a-supersingular-k3}.

We prove the following (for the converse, see Remark~\ref{rem:converse}):

\begin{theorem}\label{thm:table}
  The even hyperbolic $2$-elementary lattices of Fig.~\ref{fig:reflective}
  are reflective.  
\end{theorem}

\begin{figure}[htbp]
  \centering
  \includegraphics[width=.7\linewidth]{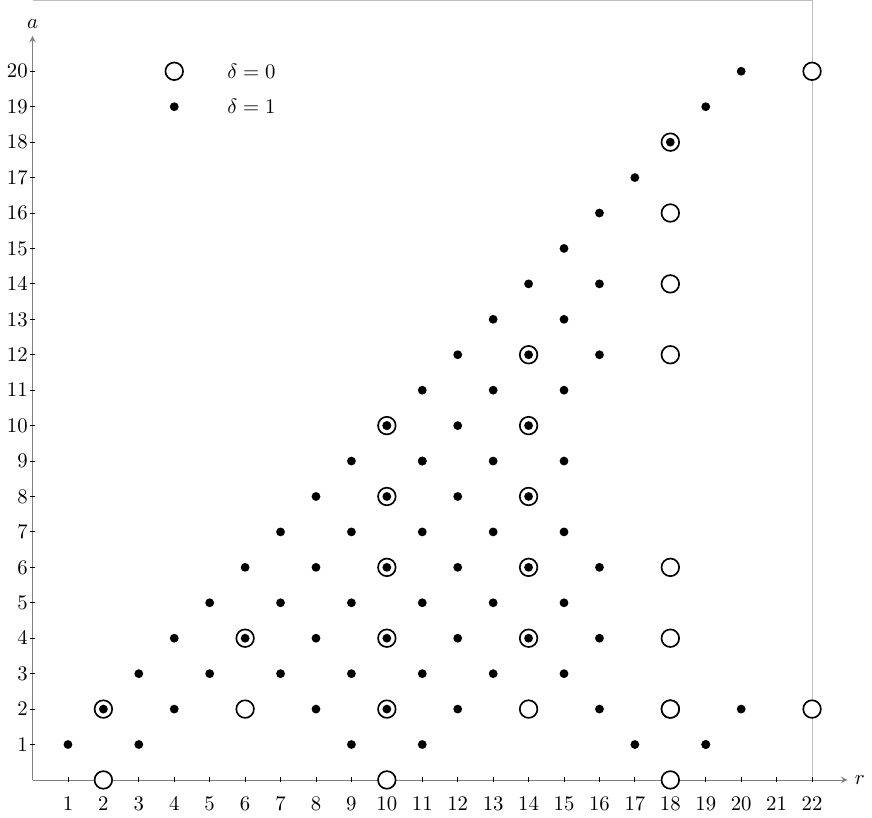}
  \caption{Reflective even hyperbolic $2$-elementary lattices}
  \label{fig:reflective}
\end{figure}

\smallskip

  It is interesting to compare the diagrams on the $r+a=22$ line with
  those on the $r+a=20$ line, computed
  in~\cite[Sec.~3]{alexeev2022mirror-symmetric}. The latter are
  smaller but less symmetric: they have the dihedral symmetry of
  $n$-gons.  Most diagrams for $r+a\le 18$ do not have any symmetries at
  all.

  % However, by~\cite[Lem.~4.14]{alexeev2022mirror-symmetric} the
  % diagram $\Gamma_r(r,a,\delta)$ of genus $g$ can be obtained from the diagram
  % $\Gamma_r'(r+1,a+1,1)$ of genus $g-1$ by a simple rule.
  % $\Aut\Gamma_r$ is the stabilizer in $\Aut\Gamma'_r$ of a
  % $(-2)$-root of divisibility~$2$.

  The diagrams on the line $r+a=20$ turn out to be key to understanding
  degenerations of K3 surfaces with a nonsymplectic involution via a
  mirror symmetry construction.  The lattices with $r+a\ge22$ do not appear as
  targets of the mirror symmetry correspondence of the ordinary K3
  surfaces. But they may appear on the mirrors of higher-dimensional
  hyperkahler manifolds.  What do the symmetries of the Coxeter diagrams
  imply about their degenerations?

\smallskip

The organization of the paper is as follows. In
Section~\ref{sec:notation} we recall some standard definitions and fix
the notations. In Section~\ref{sec:diagrams} we prove the above
theorems and give detailed descriptions of the Coxeter diagrams and
their maximal parabolic subdiagrams. In
Section~\ref{sec:automorphisms} we give some applications to the
infinite automorphism groups of K3 surfaces.

\begin{acknowledgements}
  I am grateful to Boris Alexeev for help at critical points in my
  computations.  I thank Professors Kond\=o and Mukai for useful
  discussions.  I was partially supported by the NSF grant
  DMS-2201222.
\end{acknowledgements}

\section{Definitions and notations}
\label{sec:notation}

A lattice is a group $S\simeq\bZ^r$ together with a nondegenerate
$\bZ$-valued symmetric bilinear form. It is even if $x^2$ is even for
all $x\in S$. It is $2$-elementary if $A_S:= S^*/S\simeq \bZ_2^a$ for
some $a\ge0$, where $S^*\subset S\otimes\bQ$ is the dual group.
Because we are coming from the algebraic geometry direction, for us a
hyperbolic lattice of rank $r$ has signature $(1,r-1)$.
A reflection in a vector $\alpha\in S$ is the linear transformation
$$w_\alpha(v) = v\to v -2\frac{v\cdot\alpha}{\alpha^2} \alpha.$$ A
root is a vector $\alpha\in S$ such that $\alpha^2<0$ and
$w_\alpha(S) = S$. In an even $2$-elementary lattice the roots are the
$(-2)$ vectors and the $(-4)$-vectors of divisibility~$2$. The
divisibility $\di(\alpha)$ is defined by $\alpha\cdot S = \di(\alpha)\bZ$.

The reflections in the $(-2)$-vectors generate the group $W_2(S)$, and
the reflections in all roots, in our case the $(-2)$- and
$(-4)$-roots, generate a bigger group $W_r(S)$. Both are normal
subgroups of the isometry group $O(S)$, and also of its index-$2$
subgroup $O^+(S)$ preserving the light cone.
The lattice is called
$2$-reflective, resp. reflective if these subgroups have finite index.

$A_n$, $D_n$, $E_n$ denote the standard root lattices generated by
$(-2)$-roots, and they are \emph{negative
  definite}. $U=\begin{pmatrix}0&1\\1&0\end{pmatrix}$ denotes the
hyperbolic plane. For any lattice $H$, $H(n)$ denotes the same group
with the product ${x\cdot y}$ in $H(n)$ equal to $nx\cdot y$ in~$H$. For the $2$-elementary
lattices with short $(-2)$ and long $(-4)$-roots additionally there
are root lattices $B_n(2)$, $C_n$ and $F_4$. Considered as the
$(-2)$-root lattices, they are $A_1^n$, $D_n$, and $D_4$ respectively.

\smallskip

By~\cite{nikulin1979integer-symmetric} an indefinite even $2$-elementary
lattice is uniquely determined by its signature and a triple of
integers $(r,a,\delta)$, where $r$ is its rank, $a$ is the
$\bZ_2$-rank, and $\delta\in\{0,1\}$ is an invariant which we
call coparity, see \cite[Def.~2.4, Lem.~2.5]{alexeev2022mirror-symmetric}.
A $2$-elementary lattice $S$ is coeven ($\delta=0$) if
the doubled dual $S^\dag:=S^*(2)$ is even, and it is coodd ($\delta=1$)
if $S^\dag$ is odd. A direct sum of $2$-elementary lattices is coeven
iff so is every summand. 

Any even hyperbolic $2$-elementary lattice is a direct sum of one of
the elementary hyperbolic summands $\la 2\ra$, $U$, $U(2)$ with
$(r,a,\delta) = (1,1,1)$, $(2,0,0)$, $(2,2,0)$ respectively, and
a direct sum 
of the following elliptic root lattices for the full reflection group,
which we list with their $(r,a,\delta)$.
$B_{n}(2): (n,n,1)$, $C_{4n} :(4n,2,0)$ for $4n\ge8$,
$C_{4n+2}:(4n+2,2,1)$ for $4n+2\ge6$, $F_4:(4,2,0)$, $E_7:(7,1,1)$,
$E_8:(8,0,0)$, and $E_8(2): (8,8,0)$.

% An alternative to $(r,a,\delta)$ triple of invariants is
% $(g,k,\delta)$, where $g=11-\frac12(r+a)$ is called the \emph{genus}
% and $k=\frac12(r-a)$.  Nikulin \cite{nikulin1979integer-symmetric,
%   nikulin1981quotient-groups} classified the $2$-elementary hyperbolic
% lattices that can be embedded into
% $II_{3,19} = U^{\oplus 3}\oplus E_8^{\oplus 2}$. By Torelli theorem,
% these are the $2$-elementary lattices that are Picard lattices of some
% K3 surfaces $X$ with an involution $\iota$. In this case, $g$ and $a$
% are nonnegative integers that have a clear geometric meaning: if the
% fixed locus $X^\iota$ of the involution is nonempty then it has $k+1$
% connected components and $g$ is the sum of their genera.
% For an arbitrary $2$-elementary lattice $g$ may be negative. 

\smallskip

We refer to~\cite{vinberg1972-groups-of-units} for Vinberg's theory. We
use the notations of that paper for elliptic and parabolic Coxeter
diagrams.
The vertices of $\Gamma_r$ denote the
roots $\alpha_i$ orthogonal to the facets of the fundamental
polyhedron $P_r$. The types of edges
specify the angle $\theta$ between $\alpha_i,\alpha_j$.  A single line
means $\theta=\pi/3$, double line $\theta=\pi/4$, no line
$\theta=\pi/2$, a bold line means that the hyperplanes defining the
two facets meet at an infinite point of the hyperbolic space $\cH$,
and a broken line means that they are skew.

For a $2$-elementary lattice, we denote the short, $(-2)$-roots by
white vertices and the long, $(-4)$-roots by black vertices. Then the
types of edges correspond to the following intersection numbers
$\alpha_i\cdot\alpha_j$ between the roots:

\smallskip
\begin{displaymath}
\begin{array}[h]{cc|cccc}
  \alpha_i^2 & \alpha_j^2 &\text{simple} &\text{double} &\text{bold}
  & \text{broken} \\
  \hline
  -2 & -2 & 1 & & 2 & >2 \\
  -4 & -4 & 2 & & 4 & >4 \\
   -2 & -4 &   & 2 &  & >2 \\
\end{array}
\end{displaymath}
\smallskip

The parabolic $\wA_n$, $\wD_n$, $\wE_n$ diagrams could consist either
of all short, or of all long roots.
In the latter case we denote them $\wA_n(2)$, $\wD_n(2)$, $\wE_n(2)$.
If this parabolic subdiagram
$\Gamma \subset \Gamma_r$ corresponds to a vector $v\in S$ with
$v^2=0$ then its image in $v^\perp/v$ spans an elliptic root
system. For all white vertices it is $A_n, D_n, E_n$, and for all
black vertices it is $A_n(2), D_n(2), E_n(2)$.

Similarly, for $\wB_n$, $\wC_n$ and $\wF_4$ there are two versions,
\emph{short} and \emph{long}, which are dominated by the short, resp.\
long roots. We list them in Figure~\ref{fig:extended}, together with
the elliptic root lattices $v^\perp/v$. Here, ``dominated'' means
that they form the majority of roots
``for large $n$''. For example for $\wC_n$ the short roots start to dominate
only for $n\ge4$.  One should keep in mind that $D_2=A_1^2$ and
$D_3=A_3$.

\begin{figure}[htbp]
  \centering
  \includegraphics[width=\linewidth]{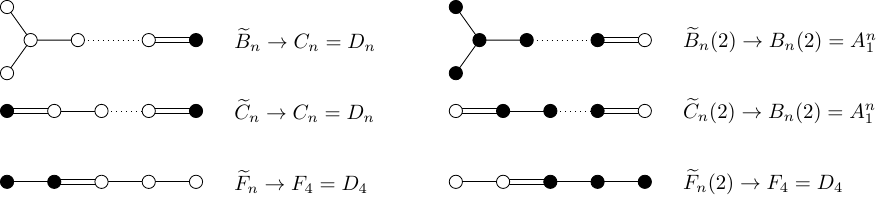}
  \caption{Extended Dynkin diagrams dominated by short or long
    vectors}\label{fig:extended} 
\end{figure}

By~\cite[Prop.~5.5]{alexeev2022mirror-symmetric} the primitive
isotropic vectors $v\in S$, $v^2=0$ are of three types.  Denoting
$\oS = v^\perp/v$ and with $v$ a vector in the first summand, the
types are:

\begin{enumerate}
\item (odd) $S = U \oplus\oS$, $a_\oS = a_S$, $\delta_\oS=\delta_S$.
\item (even ordinary) $S = U(2) \oplus\oS$, $a_\oS = a_S-2$,
  $\delta_\oS = \delta_S$. 
\item (even characteristic) $S=I_{1,1}(2)\oplus\oS$, $a_\oS = a_S-2$,
  $\delta_S = 1$ and $\delta_\oS=0$.
\end{enumerate}

Thus, classifying isotropic vectors in an even hyperbolic
$2$-elementary lattice with the invariants $(r,a,\delta)$ is
equivalent to classifying the even negative definite lattices $\oS$
with the invariants $(r-2,a,\delta)$, $(r-2,a-2,\delta)$ and $(r-2,a-2,0)$.

% If $R$ is the root sublattice of $\oS$ for the full
% reflection group $W_r(\oS)$ then we denote $\oS=R*$ if
% $|\oS/R|<\infty$ and $\oS=R\!*\!*$ if $|\oS/R| = \infty$.

\section{Coxeter diagrams and maximal parabolic subdiagrams}
\label{sec:diagrams}

\begin{proof}[Proof of Thm.~\ref{thm:table}]
  For some of these lattices the Coxeter diagrams are well known: for
  those on the $a=0$ line (they are unimodular) and for those on the
  $r=a$ line (they are twice the unimodular). The two hardest cases
  were done by Vinberg and Kaplinskaja \cite{vinberg1978the-groups}.
  
  For others we performed the computation following
  Vinberg's algorithm. Each of these lattices can be written as
  $U\oplus \Lambda$ or $U(2)\oplus \Lambda$ for some root lattice
  $\Lambda$, where $U$ is the hyperbolic plane.  We chose the control
  vector to be a vector $v_0$ in the first summand, with $v_0^2=2$ for
  $U$, resp.\ $v_0^2=4$ for $U(2)$, and ran the algorithm using a
  custom Sage~\cite{sagemath} script written for this purpose.

  We checked the completeness of the diagrams directly, by confirming
  that the cones in the Minkowski spaces $\bR^{1,r-1}$
  defined by the roots of $\Gamma_r$ lie entirely in~$\overline{\cC}$.

  For the diagrams without broken edges the easy sufficient condition
  of~\cite{vinberg1973some-arithmetic} also works to verify the
  completeness by hands: the diagrams do not
  contain Lann\'er subgraphs, and every connected parabolic subdiagram
  is contained in a maximal parabolic subdiagram of maximal rank. For
  the diagrams with broken edges, the criterion 
  of~\cite[Prop.~2]{vinberg1972-groups-of-units} works.
\end{proof}

\begin{proof}[Proof of Thm.~\ref{thm2}]
  This follows by observation, analyzing the Coxeter diagrams produced
  by the above method.
\end{proof}

\begin{proof}[Proof of Thm.~\ref{thm1}]
  For the nonreflective lattices the proof is as follows. One has
  \begin{displaymath}
    (19,3,1) = \la-2\ra\oplus (18,2,1)
    \text{ and }
    (18,4,1) = \la-2\ra\oplus (17,3,1).
  \end{displaymath}
  The fundamental polyhedron $P_r(18,2,1)$ is a face of
  $P_r(19,3,1)$ and $\Gamma_r(18,2,1)$ is smaller than
  $\Gamma_r(19,3,1)$, 
  see~\cite[Lem.~4.14]{alexeev2022mirror-symmetric}. Similarly for
  $(18,4,1)$.  The diagrams $\Gamma_r(18,2,1)$, $\Gamma_r(17,3,1)$
  are infinite by~\cite[Thm.~3.6]{alexeev2022mirror-symmetric}, and so
  are $\Gamma_r(19,3,1)$ and $\Gamma_r(18,4,1)$.

  Finally, one has $(17,5,1)\simeq U(2)\oplus\Lambda$ for
  any even $2$-elementary negative definite lattice $\Lambda$ with
  $(r,a,\delta)=(15,3,1)$, where $U$ is a hyperbolic plane. If the lattice
  $(17,5,1)$ were reflective then $\Lambda$ would correspond to a
  maximal parabolic subdiagram of $\Gamma_r$ and the root sublattice
  $R(\Lambda)$ would have rank equal to $r(\Lambda)=15$.  But
  by~\cite[Table~2]{alexeev2022mirror-symmetric} there exists $\Lambda$
  with $R(\Lambda)=A_{13}A_1(2)$ of rank~$14$.
\end{proof}

\begin{remark}\label{rem:converse}
  From the experimental data, it appears that the list of reflective
  $2$-elementary
  lattices given in Fig.~\ref{fig:reflective} may be complete. For many, but
  not all of the missing nodes the non-reflectivity can be checked by using
  %Borcherds' method \cite{borcherds1987automorphism-groups},
  the duality
  $L\to L^\dag$ and the arguments as in the proof of Theorem~\ref{thm1}.
\end{remark}

Next, we list the Coxeter diagrams of Fig.~\ref{fig:reflective} with
$r+a\ge22$ and the maximal parabolic diagrams in them.
The diagrams for the lattices $(11,11,1)$, $(12,10,1)$, $(13,9,1)$,
$(14,8,0)$ and $(14,8,1)$ are shown in Figure~\ref{fig:coxeter}.

\begin{figure}[htbp]
  \centering
  \includegraphics[width=\linewidth]{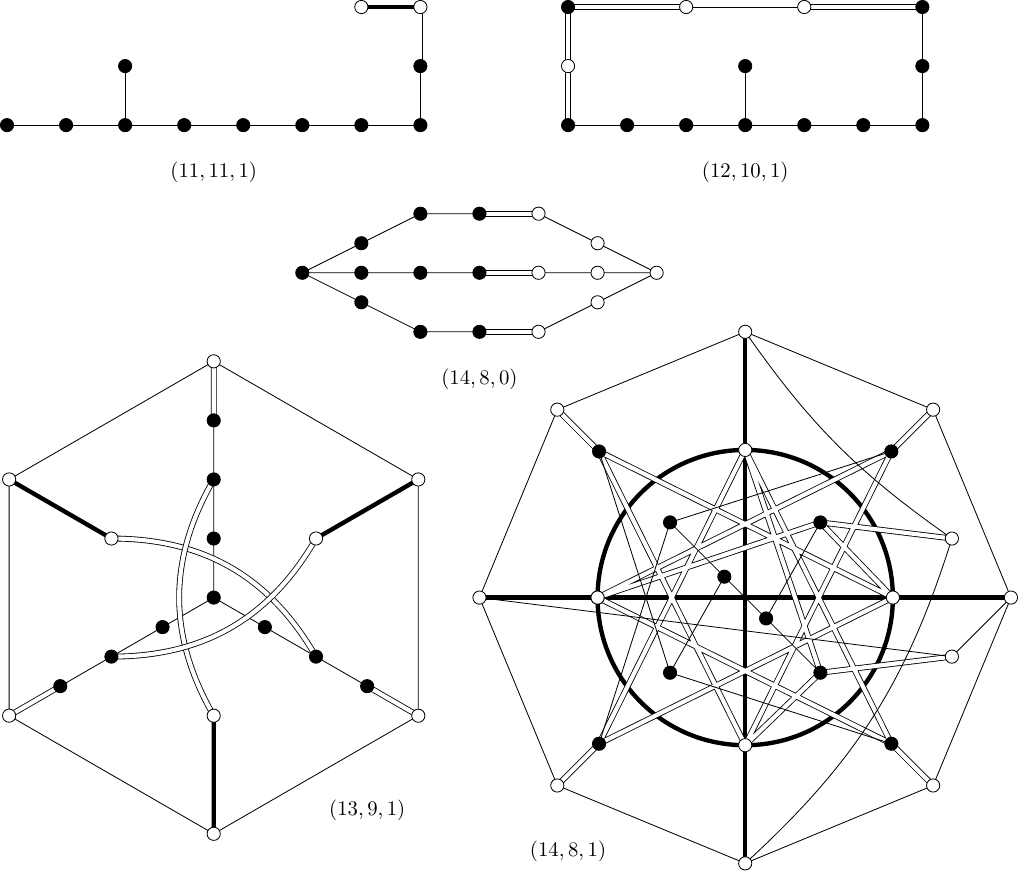}
  \caption{Coxeter diagrams $\Gamma_r$ for ranks $r\le 14$}\label{fig:coxeter}
\end{figure}

\subsection*{\boldmath$(11,11,1)$\unboldmath}\label{sec:11-11-1}

% This lattice is isomorphic to
% \begin{math}
%   U(2) \oplus E_8(2) \oplus A_1 = I_{1,10}(2).
% \end{math}
%The Coxeter diagram
$\Gamma_r$ has $12$ roots.  There are two isotropic vectors $v$ modulo
$O^+(S)$. They correspond to the following subdiagrams of the Coxeter diagram,
both even ordinary: $\wE_8(2) \wA_1$ and $\wB_9(2)$.
The automorphism group of  $\Gamma_r$ is trivial.

\subsection*{\boldmath$(12,10,1)$\unboldmath}\label{sec:12-10-1}

% This lattice is isomorphic to
% \begin{math}
%   U\oplus E_8(2) \oplus B_2(2).
% \end{math}
%The Coxeter diagram
$\Gamma_r$ has $14$ roots. Note that for the parabolic
subdiagram $\wC_3$ the corresponding root lattice $v^\perp/v$ is
$C_3=D_3=A_3$.
There are $5$ maximal parabolic subdiagrams listed below.  The
automorphism group of
$\Gamma_r$ is trivial.

\begin{enumerate}
\item odd: $\wE_8(2)\wC_2(2)$, $\wC_{10}(2)$.
\item even ordinary: $\wE_7(2)\wC_3$, $\wB_6(2)\wF_4(2)$, $\wB_8(2)\wC_2$.
\end{enumerate}

\subsection*{\boldmath$(13,9,1)$\unboldmath}\label{sec:13-9-1}

% This lattice is isomorphic to
% \begin{displaymath}
%   U\oplus F_4 \oplus B_7(2) = U\oplus D_4 \oplus A_1^7.
% \end{displaymath}
% when considered as a root lattice for the $(-2)$ and $(-4)$-roots,
%or for the $(-2)$-roots only, respectively.
%The Coxeter diagram
$\Gamma_r$ has
$19$ roots.  There are $22$ maximal parabolic subdiagrams, and $7$
modulo $\Aut\Gamma_r=S_3$:

\begin{enumerate}
\item odd: $\wC_8(2)\wC_3(2)$, $\wE_7(2)\wB_3\wA_1$, $\wC_7(2)\wF_4$.
\item even ordinary: $\wF_4^2(2)\wB_3(2)$, $\wC_6\wB_5(2)$, $\wB_6(2)\wC_4\wA_1$
  $\wE_6(2)\wA_5$.
\end{enumerate}

It is clear that this diagram is built on top of the graph $K_3^{(2)}$. The
main roots are the $6$ roots on the outside hexagon, and the additional roots
come in four layers culminating with the central vertex. The roots
generate the lattice, so the $S_3$-action on $\Gamma_r$ extends to the
action on the lattice $H$ itself.

\subsection*{\boldmath$(14,8,0)$\unboldmath}\label{sec:14-8-0}
% This lattice is isomorphic to
% \begin{math}
%   U(2) \oplus F_4^3 = U(2)\oplus D_4^3.
% \end{math}
%The Coxeter diagram
$\Gamma_r$ has $17$ roots. There are $11$ maximal parabolic
subdiagrams, and $5$ modulo $\Aut\Gamma_r=S_3$:

\begin{enumerate}
\item odd: $\wE_7(2)\wB_5$, $\wC_8(2)\wF_4$.
\item even ordinary: $\wF_4^3(2)$, $\wC_6\wB_6(2)$, $\wE_6\wE_6(2)$.  
\end{enumerate}

It is clear that this diagram is built on top of the graph
$D_4^{(2)}$, where $D_4$ is a tree with four vertices, the central
vertex of degree $3$ and three ends.

\subsection*{\boldmath$(14,8,1)$\unboldmath}\label{sec:14-8-1}

% This lattice is isomorphic to
% \begin{math}
%   U \oplus F_4^2 \oplus B_4(2) = U\oplus D_4^2 \oplus A_1^4.
% \end{math}
% For the choice of the control vector $v_0\in U$ with $v_0^2=2$,
% Vinberg's algorithm ends in $17$ steps.
%The Coxeter diagram
$\Gamma_r$ has $24$
roots. There are $127$ maximal parabolic subdiagrams, and $15$ modulo
$\Aut\Gamma_r=S_4$:

\begin{enumerate}
\item odd: $\wF_4^2\wC_4(2)$, $\wB_6\wC_6(2)$,
  $\wC_6(2)\wB_4\wC_2(2)$, $\wE_6(2)\wA_5\wA_1$, $\wD_6(2)\wB_3^2$,
  $\wC_4^3(2)$.
\item even ord: $\wC_6\wF_4(2)\wC_2$, $\wC_4^2\wB_4(2)$,
  $\wD_6\wB_5(2)\wA_1$, $\wE_7\wB_5(2)$, $\wC_8\wB_4(2)$,
  $\wA_7\wD_5(2)$.
\item even char: $\wF_4^3(2)$, $\wC_6\wB_6(2)$, $\wE_6\wE_6(2)$.
\end{enumerate}

\smallskip

It is clear that this diagram is built on top of the graph
$K_4^{(2)}$. Denote the main roots $\alpha^m_i$ with $i=1,2,3,4$ for
the $4$ vertices and $\alpha^m_{ij}$ with $1\le i<j\le 4$ for the $6$
edges of $K_4$.
%For each vertex there is an additional $(-4)$-root
%$\alpha_i$ with $\alpha_i^m\alpha_i=2$.
For each edge there is an
additional $(-2)$-root $\alpha_{ij}$ attached to it, and a $(-4)$-root
$\alpha_{ij}^+$ attached to $\alpha_{ij}$. One has
$\alpha_{ij}^m\cdot\alpha_{ij} = \alpha_{ij}\cdot\alpha^+_{ij}=2$, 
$\alpha_{ij}^m\cdot\alpha^+_{ij} = 0$.
Let $\bar\alpha_{ij} = \alpha_{ij} + \alpha^+_{ij}$.

Then the $12$ roots $\alpha_{ij}$, $\bar\alpha_{ij}$ are mutually
orthogonal, and are also orthogonal to the central $(-4)$-root
$\alpha_{cen}$. The orthogonal complement of these $13$ vectors in $H$
is $\bZ h$ with
$h= \tfrac12(\sum \alpha_i^m + \sum\alpha_{ij} + \sum\bar\alpha_{ij})$.
One has $h^2=4$. Set $e=\tfrac12(h+\alpha_{cen})$ and
$f=\tfrac12(h-\alpha_{cen})$. Then the $14$ vectors $\alpha_{ij}$,
$\bar\alpha_{ij}$, $e$, $f$ form the standard basis of a sublattice of
$H$ isomorphic to 
$(I_{0,12}\oplus U)(2)$. Thus, $H$ can be identified with an
$S_4$-invariant lattice lying between it and its dual
$(I_{0,12}\oplus U)(\tfrac12)$.

Picking $v_0=h$ to be the control vector, Vinberg's algorithm ends in
two steps. 

\subsection*{\boldmath$(15,7,1)$\unboldmath}\label{sec:15-7-1}

% This lattice is isomorphic to
% \begin{displaymath}
%   U\oplus F_4^3\oplus A_1 = U\oplus D_4^3\oplus A_1.
% \end{displaymath}
%The Coxeter diagram
$\Gamma_r$ contains $60$ roots.  As in
Definition~\ref{def:common-graph}, it is build on top of the
graph $K_5^{(2)}$. The subdiagram of the main roots is $K_5^{(2)}$, as shown
on the left in Fig.~\ref{fig:15roots}. The $15$ main roots form a
basis of the lattice.

\begin{figure}[htbp]
  \centering
  \includegraphics[width=\textwidth]{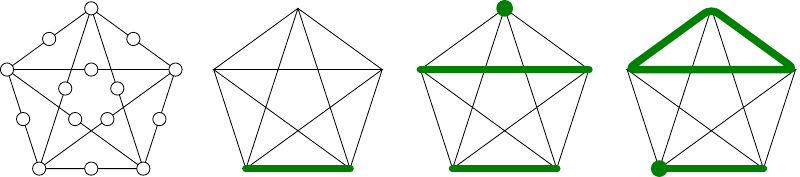}
  \caption{Additional roots for $(15,7,1)$}\label{fig:15roots}
\end{figure}

Recall Definition~\ref{def:nth-graph} of the graph $G^{(2)}$. The
vertices of $G^{(2)}$ are naturally in a bijection with the set of
vertices + edges of $G$. Using this bijection, we identify the
additional roots $\alpha_C$ (see
Definition~\ref{def:common-graph}) with certain collections $C$ of
vertices and edges of $G$. We use this notation for the other graphs
below as well.

For $G=K_5$ these collections $C$ of vertices and edges are
pictured in Fig.~\ref{fig:15roots}:

\begin{enumerate}
\item $10$ $(-4)$-roots with $C$ an edge, e.g.\ $C=\{12\}$.
\item $15$ $(-2)$-roots with $C$ a vertex and two edges, all disjoint
  from each other, e.g.\ $C=\{1,23,45\}$.
\item $20$ $(-2)$-roots with $C=\{1,12,34,35,45\}$ etc, with $C$
  consisting of three edges in a triangle, an edge disjoint from it,
  plus a vertex on the latter edge.
\end{enumerate}
% We can specify the collections $C$ by marking some vertices of the
% subdiagram $G^{(2)}$ of the main roots. Equivalently, and easier to
% picture, we can specify them by marking some vertices and edges of the
% initial graph $G=K_5$.

The edges between the main and additional roots are specified in
Definition~\ref{def:common-graph}. We now list the typical edges
between the additional roots $\alpha_C$, $\alpha_{C'}$. The others
follow by $S_5$-symmetry.

\begin{enumerate}

\item $\alpha_{12}$ is connected by a single line to $\alpha_{34}$; by
  a double line to $\alpha_{3,14,25}$; and by a broken line to
  $\alpha_{3,13,24,25,45}$.

\item $\alpha_{1,23,45}$ 
  is connected by a double line to
  $\alpha_{24}$;
  a bold line to $\alpha_{1,24,35}$, $\alpha_{2,13,45}$,
  $\alpha_{1,12,34,35,45}$ and $\alpha_{2,14,15,23,45}$;
  and a broken line to $\alpha_{2,14,35}$,
  $\alpha_{2,12,34,35,45}$ and $\alpha_{2,13,14,25,34}$.
  
\item $\alpha_{1,12,34,35,45}$
  a bold line to $\alpha_{1,23,45}$ and $\alpha_{3,12,45}$;
  and a broken line to
  $\alpha_{23}$,
  $\alpha_{2,13,45}$, $\alpha_{3,14,25}$,
  $\alpha_{1,13,24,25,45}$,
  $\alpha_{2,12,34,35,45}$, 
  $\alpha_{3,12,15,25,34}$, 
  $\alpha_{3,14,15,23,45}$, \linebreak
  $\alpha_{2,14,15,23,45}$, and
  $\alpha_{3,13,24,25,45}$. 
\end{enumerate}

The entire intersection matrix can be easily described as follows. The
main roots $\alpha^m_1$, $\alpha^m_{12}$ etc.\ form a basis of the
lattice $S$.
Let $\omega_1$ and $\omega_{12}$ etc.\ denote the vectors
such that $\alpha^m_I\cdot \omega_J = 2\delta_{IJ}$, i.e.\ forming twice the
dual basis. Then $\alpha_C = \sum_{I\in C}\omega_I$, where $I$ is one
of the sets $1, \dotsc, 5, 12, \dotsc, 45$.  Thus, it suffices to specify the
typical intersection numbers between $\omega_I$. They are:
\begin{displaymath}
  \omega_1^2 = -6, \ \omega_1\cdot\omega_2=2,\
  \omega_1\cdot\omega_{12}=-2,\ \omega_1\cdot\omega_{23}=2,\
  \omega_{12}^2=-4,\ \omega_{12}\cdot\omega_{23}=0,\ \omega_{12}\cdot\omega_{34}=2.
\end{displaymath}

The $(-4)$-vectors of $\Gamma_r$ form the
Petersen graph on $10$ vertices, with simple edges.

\smallskip There are $1027$ maximal parabolic subdiagrams, and $20$
modulo $\Aut\Gamma_r=S_5$:

\begin{enumerate}
\item odd: $\wE_7\wC_6(2)$, $\wB_6\wF_4\wC_3(2)$, $\wB_6\wC_6(2)\wA_1$, 
  $\wD_6\wC_5(2)\wC_2(2)$, $\wB_8\wC_5(2)$, $\wA_5\wA_5(2)\wB_3$, 
  $\wF_4^3\wA_1$, $\wC_4(2)\wB_4^2\wA_1$, 
  $\wE_6\wE_6(2)\wA_1$, $\wA_7\wD_5(2)\wA_1$, $\wD_4\wC_3^3(2)$.
\item even ordinary: $\wE_7\wB_5(2)\wA_1$, $\wD_8\wB_4(2)\wA_1$, 
  $\wE_8\wB_5(2)$, $\wC_8\wF_4(2)\wA_1$, $\wE_7\wF_4(2)\wC_2$, \linebreak
  $\wD_6\wC_4\wB_3(2)$, $\wA_9\wA_4(2)$, $\wC_6^2\wA_1$, $\wC_{10}\wB_3(2)$.
\end{enumerate}

\subsection*{\boldmath$(16,6,1)$\unboldmath}\label{sec:16-6-1}

% This lattice is isomorphic to
% \begin{displaymath}
%   U(2) \oplus C_8\oplus C_6 = U(2)\oplus D_8 \oplus D_6.
% \end{displaymath}
% The Coxeter diagram
$\Gamma_r$ contains $118$ roots. The starting graph $G$ is
$K_6$, a complete graph on $6$ vertices. As in
Definition~\ref{def:common-graph}, the subdiagram of the main roots is
$K_6^{(2)}$.
It has $6$ $(-2)$-roots corresponding to vertices of
$K_6$---call them $1$, $2$, $3$, $4$, $5$, $6$---and $15$ $(-2)$-roots
corresponding to the edges of $K_6$---call them $12$, $13$ etc, for a
total of $21$ main roots. Let us denote the main roots $\alpha^m_1$,
$\alpha^m_{12}$ etc.

\begin{figure}[htbp]
  \centering
  \includegraphics[width=.85\linewidth]{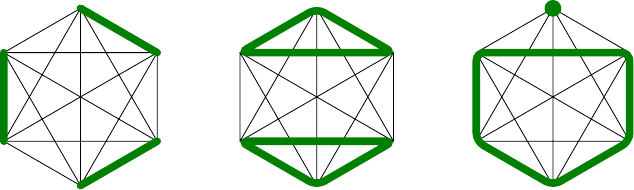}
  \caption{Additional roots for $(16,6,1)$}\label{fig:16roots}
\end{figure}
The additional roots $\alpha_C$ correspond to the following
collections of vertices and edges of the starting graph $G$, pictured
in Fig.~\ref{fig:16roots}, modulo $S_6$:

\begin{enumerate}
\item $15$ $(-4$)-roots $\alpha_{12,34,56}$ for $C$ equal to a triple of
  disjoint edges. 
\item $10$ $(-4)$-roots $\alpha_{12,23,31,45,56,64}$ for $C$ equal to
  the $6$ edges in two disjoint triangles.
\item $72$ $(-2)$-roots   $\alpha_{12,23,34,45,51,6}$ for 
  a cycle of $5$ edges and a vertex.
\end{enumerate}

The $15$ main roots $\alpha_{12}$ for the edges, together with
$\alpha^m=\alpha^m_1 + \alpha^m_2 + \dotsb + \alpha^m_6$ span the lattice over
$\bQ$. Let $\omega_{12},\dotsc,\omega_{56}, \omega$ be twice the dual
basis. Then the additional roots are $\alpha_C = \sum_{I\in C}
\omega_I$, where we formally set $\omega_1 = \omega$ etc for the
vertices. The multiplication table for this dual basis is very easy:
$\omega_{12}^2 = -\frac{16}{9}$ for the edges, and all the other
products and squares are $\frac29$. Thus,
\begin{displaymath}
  \omega_C\cdot\omega_{C'} = \frac29|C|\cdot|C'|
  - 2|C\cap C'\cap{\rm Edges}(G)|.
\end{displaymath}
Explicitly, the edges between the additional roots are:
\begin{enumerate}
\item Two roots of type (1) are joined by a bold line if they don't
  share an edge.
\item Two roots of type (2) are joined by a bold line.
\item Two roots of type (3) are joined by a bold line if they share
  $3$ edges, and by a broken line otherwise.
\item Roots of types (1) and (2) are joined by a bold line if they
  don't share an edge, by a single line if they share one edge.
\item Roots of types (1) and (3) are joined by a broken line if they
  don't share edges, and by a double line if they share one edge.
\item Roots of types (2) and (3) are joined by a double line if they
  share three edges, and by a broken line if they share fewer edges.
\end{enumerate}

\smallskip
The $15$ additional roots $\alpha_{12}$ etc.\ for the edges are
mutually orthogonal $(-2)$-roots. Denote by
$v=\alpha^m_1 + \dotsb + \alpha^m_6$ the sum of the main roots for the
vertices of $K_6$ and by $e = \alpha^m_{12} + \dotsb + \alpha^m_{56}$
the sum of the main roots for the edges of $K_6$. Then
$h=\frac13(v+e)$ is a vector in $H$ satisfying $h^2=2$ and orthogonal
to the above $15$ roots. Then $h, \alpha_{12}, \dotsc, \alpha_{56}$
form a sublattice of $H$ isomorphic to $I_{1,15}(2)$. Thus, $H$ can be
identified with an $S_6$-symmetric lattice lying between it and its dual
$I_{1,15}(\tfrac12)$. 

With the control vector $v_0=h$, Vinberg's algorithm terminates in
four steps. 

\medskip

There are $8917$ maximal parabolic subdiagrams, 
$28$ modulo $\Aut\Gamma_r = S_6$:
\begin{enumerate}
\item odd: $\wB_6\wC_4(2)\wB_4$, $\wB_{10}\wC_4(2)$,
  $\wB_6^2\wC_2(2)$, $\wE_8\wC_6(2)$, $\wD_6\wB_4\wC_3(2)\wA_1$,
  $\wD_8\wC_4(2)\wC_2(2)$, $\wE_6\wA_5(2)\wB_3$, $\wB_8\wF_4\wC_2(2)$,
  $\wB_8\wC_6(2)$, $\wB_6\wF_4^2$, $\wE_7\wF_4\wC_3(2)$,
  $\wE_7\wC_5(2)\wC_2(2)$, \linebreak
  $\wD_6\wC_3^2(2)\wC_2(2)$,
  $\wD_4^2\wC_2^3(2)$, $\wA_5^2\wA_2^2(2)$,
  $\wA_7\wA_3(2)\wB_3\wA_1(2)$, $\wB_4^3\wC_2(2)$,
  $\wA_9\wA_4(2)\wA_1$.
\item even ordinary: $\wE_8\wF_4(2)\wC_2$, $\wD_{10}\wB_3(2)\wA_1$,
  $\wC_{10}\wF_4(2)$, $\wE_7\wC_4\wB_3(2)$, $\wC_{12}\wC_2$,
  $\wD_8\wC_4\wC_2$, $\wA_{11}\wA_2(2)\wA_1(2)$, $\wE_7\wC_6\wA_1$,
  $\wC_8\wC_6$, $\wD_6^2\wC_2$.
\end{enumerate}

Prof. Shigeru Mukai has informed me that this Coxeter diagram can also
be obtained as a face of the fundamental polyhedron for the
lattice $(A_3 \oplus A_1^6)^\perp$ in $II_{1,25}$ which was
computed by Kond\=o \cite{kondo1998the-automorphism-group} using the 
Borcherds method \cite{borcherds1987automorphism-groups}.

\subsection*{\boldmath$(18,4,0)$\unboldmath}\label{sec:18-4-0}

% This lattice is isomorphic to
% \begin{displaymath}
%   U\oplus C_8^2 = U\oplus D_8^2.
% \end{displaymath}
%The Coxeter diagram
The diagram $\Gamma_r$ was described by Keum and Kond\=o
\cite{keum2001the-automorphism-group} using the method of
\cite{borcherds1987automorphism-groups} and the Leech lattice.  Our
description is quite different.  The Borcherds method does not always
produce the Coxeter diagram of a reflexive lattice (e.g. it does not
in the $(16,6,1)$ case). So the equivalence between our diagram and
Keum-Kond\=o's is not immediate. But it can be checked directly.

\begin{figure}[htbp]
  \centering
  \includegraphics{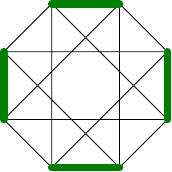}
  \caption{Additional roots for $(18,4,0)$}\label{fig:18roots}
\end{figure}

The Coxeter diagram $\Gamma_r$ contains $48$ roots. The starting graph is
$G=K_{4,4}$, the complete bipartite graph on $8$ vertices, shown in
Figure~\ref{fig:18roots}. 
As in Definition~\ref{def:common-graph}, the subdiagram of the main
roots is $K^{(2)}_{4,4}$. It has $8$ $(-2)$-roots corresponding to the
vertices of $K_{4,4}$ and $16$ $(-2)$-roots corresponding to the edges
of $K_{4,4}$, for a total of $24$ main roots.

The additional $24$ $(-4)$-roots $\alpha_C$ are in bijection with the
sets of four disjoint edges in $G$, i.e.\ with the perfect matchings on
the set of vertices of $G$. The intersection numbers are
\begin{math}
  \alpha_C\cdot\alpha_{C'} = 4 - 2|C\cap C'|.
\end{math}
Thus, $\alpha_C$ and $\alpha_{C'}$ are connected by a bold line if
they don't share any edges, and by a single line if they share exactly
one edge.
Via a natural bijection between the perfect matchings in
$K_{4,4}$ and elements of the group $S_4$, 
$C\cap C'$ is the set of fixed points of $C\inv\circ C'\in S_4$.

\smallskip

The graph $K_{4,4}$ is bipartite, its vertices are split into two
groups of four.  Denote by $v^{(1)}$, resp.\ $v^{(2)}$, the sum of the
main roots in the first, resp.\ the second group. Denote by $e$ the sum
of the main roots for the $16$ edges. Then $a=\tfrac14(2v^{(1)} + e)$
and $b=\tfrac14(2v^{(2)} + e)$ are two vectors in $H$ satisfying
$a^2=b^2=0$, $a\cdot b=4$. They are also orthogonal to the $16$ main
$(-2)$-roots $\alpha^m_{ij}$ for the edges, which are mutually
orthogonal as well. Together they form a standard basis of the
sublattice $(I_{0,16}\oplus U)(2)$ in $H$. Thus, $H$ can be identified
with an $(\Aut K_{4,4})$-invariant lattice lying between it and its dual
$(I_{0,16}\oplus U)(\tfrac12)$.

With the control vector $v_0=a+b$, Vinberg's algorithm terminates in
two steps.

\medskip

The automorphism group of the diagram is $\Aut\Gamma_r = \Aut K_{4,4}
= S_2 \ltimes (S_4\times S_4)$. 
There are $5244$ maximal parabolic subgraphs, $17$
modulo $\Aut\Gamma_r$:
\begin{enumerate}
\item odd: $\wD_8\wB_4^2$, $\wE_8\wF_4^2$, $\wE_7\wB_6\wC_3(2)$,
  $\wA_{11}\wB_3\wA_2(2)$, $\wE_6^2\wA_2^2(2)$, $\wB_{12}\wF_4$,
  $\wD_6^2\wC_2^2(2)$, $\wD_{10}\wC_3^2(2)$, $\wB_8^2$,
  $\wA_7^2\wA_1^2(2)$, $\wD_4^4$.
\item even ordinary: $\wA_{15}\wA_1(2)$, $\wE_7^2\wC_2$, $\wD_{12}\wC_4$,
  $\wC_8\wE_8$, $\wC_{16}$, $\wD_8^2$.
\end{enumerate}

\subsection*{\boldmath$(20,2,1)$\unboldmath}\label{sec:20-2-1}

% This lattice is isomorphic to
% \begin{displaymath}
%   U\oplus E_8^2 \oplus B_2(2) = U\oplus E_8^2 \oplus A_1^2.
% \end{displaymath}
The Coxeter diagram was computed by Vinberg and Kaplinskaja 
in~\cite{vinberg1978the-groups}, see also~\cite{vinberg1983two-most}.
For easier comparison with the above diagrams we list the
roots in Figure~\ref{fig:20roots}.  There are $25$ main
roots for the vertices and edges of the Petersen's graph
and $25$ additional roots: $20$ $(-2)$-roots of the first kind and $5$
$(-4)$-roots of the second kind.

\begin{figure}[htbp]
  \centering
  \includegraphics[width=.6\linewidth]{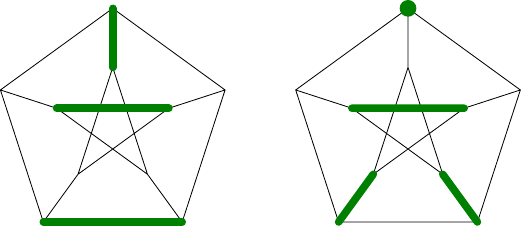}
  \caption{Additional roots for $(20,2,1)$}\label{fig:20roots}
\end{figure}

There are $581$ maximal parabolic subdiagrams, $13$ modulo
$\Aut\Gamma_r=S_5$, all odd: $\wB_{18}$, $\wB_{10}\wE_8$,
$\wD_{16} \wC_2(2)$, $\wD_{10}\wE_7\wA_1$, $\wE_8^2\wC_2(2)$,
$\wE_7^2\wB_4$, $\wD_{12}\wB_6$, $\wD_8^2\wC_2(2)$, $\wA_9^2$,
$\wA_{15}\wB_3$, $\wA_{11}\wE_6\wA_1(2)$, $\wA_{17}\wA_1$, $\wD_6^3$.
On the unique K3 surface with this Picard lattice they define $13$
types of elliptic pencils with a section. The rules of
Fig.~\ref{fig:extended} lead to a description of the Kodaira types of
the singular fibers in these pencils.

\subsection*{\boldmath$(13,11,1)$\unboldmath}\label{sec:13-11-1}
$\Gamma_r$ has $16$ roots. There are $9$ maximal parabolic
subdiagrams, $5$ modulo $\Aut\Gamma_r=S_2$:
\begin{enumerate}
\item odd: $\wC_{11}(2)$, $\wE_8(2) \wC_3(2)$, 
\item even ord: $\wB_8(2) \wB_3(2)$, $\wB_7(2) \wF_4(2)$, $\wE_7(2)
  \wC_3 \wA_1$. 
\end{enumerate}

\subsection*{\boldmath$(14,10,1)$\unboldmath}\label{sec:14-10-1}
$\Gamma_r$ has $21$ roots, There are $69$ maximal
parabolic subdiagrams, $12$ modulo $\Aut\Gamma_r = D_6$:
\begin{enumerate}
\item odd: $\wC_8(2) \wF_4$, $\wE_7(2) \wB_3 \wC_2(2)$, $\wC_8(2)
  \wC_4(2)$, $\wD_8(2) \wC_4(2)$, 
\item even ord: $\wB_4(2) \wB_4(2) \wB_4(2)$, $\wC_6 \wB_6(2)$,
  $\wB_6(2) \wC_4 \wC_2$, $\wB_4(2) \wF_4(2) \wF_4(2)$,
  \linebreak
  $\wD_6(2)
  \wC_3 \wC_3$, $\wE_6(2) \wA_5 \wA_1$,
\item even char: $\wE_7(2) \wC_5$, $\wB_8(2) \wF_4(2)$.
\end{enumerate}

\subsection*{\boldmath$(15,9,1)$\unboldmath}\label{sec:15-9-1}
% This lattice is isomorphic to
% \begin{displaymath}
%   U \oplus F_4^2 \oplus B_5(2) = U \oplus D_4^2 \oplus A_1^5.
% \end{displaymath}

$\Gamma_r$ has $98$ roots which modulo $\Aut\Gamma_r=\PGL(2,7)$ of
order $336$ split into four orbits of sizes $14$, $14$, $14$,
$56$. The first group of $14$ forms the subdiagram $G$ of the main
roots. Here, $G$ is a $4$-regular bipartite graph on $14$ vertices. It
has a natural embedding into the complete bipartite graph
$K_{7,7}$. $G$ is the co-Heawood graph; its edges
in $K_{7,7}$ are complementary to the
edges of the well known Heawood graph, a $3$-regular graph on $14$
vertices. One way to describe $G$ is to say that its vertices
correspond to the points and lines in the Fano plane $\bP^2$ over
$\bF_2$, and that a line and a point are joined by an edge iff they
are \emph{not} incident.

The remaining roots are additional. They are connected to the main
roots by $1$, $4$ and $6$ edges.  The first of these groups is formed
by the $(-4)$-roots. This subdiagram is isomorphic to the Heawood
graph.  Each of these roots is connected to exactly one main root and
$\alpha_i^m\alpha_i = 4$.  The remaining two orbits are formed by the
$(-2)$-roots of
divisibility~$2$. By~\cite[Lem.~4.14]{alexeev2022mirror-symmetric}
removing one of these roots and its $(-2)$-neighbors gives the Coxeter
diagrams for the lattices $(14,8,0)$ and $(14,8,1)$.

There are $2114$ maximal parabolic subdiagrams, and $20$ modulo
$\Aut\Gamma_r=\Aut G$:
\begin{enumerate}
\item odd: $\wD_6(2)\wB_3^2\wA_1$, $\wA_7(2)\wA_3\wC_3(2)$,
  $\wE_6(2)\wA_5\wC_2(2)$, $\wC_6(2)\wB_4\wC_3(2)$, $\wC_4^3(2)\wA_1$,
  \linebreak
  $\wE_7(2)\wB_5\wA_1$, $\wC_8(2)\wF_4\wA_1$, $\wC_7(2)\wB_6$,
  $\wC_5(2)\wF_4^2$.
\item even ord: $\wA_7\wD_5(2)\wA_1$, $\wE_6\wE_6(2)\wA_1$,
  $\wC_4^2\wB_4(2)\wA_1$, $\wD_4\wB_3^3(2)$, $\wF_4^3(2)\wA_1$,
  $\wC_8\wB_5(2)$, $\wA_5\wA_5(2)\wC_3$, $\wE_7\wB_6(2)$,
  $\wC_6\wF_4(2)\wB_3(2)$, $\wD_6\wB_5(2)\wC_2$, $\wC_6\wB_6(2)\wA_1$. 
\end{enumerate}

Note that
$(15,9,1) \simeq \la -2\ra \oplus (14,8,1) \simeq \la -2\ra \oplus
(14,8,0)$.  Among these two, $(14,8,1) = S_X$ for some K3 surface with
a nonsymplectic involution, and $(15,9,1)$ is the Picard lattice of
the hyperkahler manifold $\Hilb^2(X)$.
The symmetry group of the Coxeter diagram $\Gamma_r(14,8,1)$ is only $S_4$,
so in this case $\Hilb^2(X)$ is much more symmetric than the K3
surface $X$.

\subsection*{\boldmath$(18,6,0)$\unboldmath}\label{sec:18-6-0}

The Coxeter diagram $\Gamma_r$ was described by Borcherds
\cite{borcherds2000reflection-groups} and it consists of $64+896$
roots.
I computed the maximal parabolic subdiagrams of $\Gamma_r$. There are
$90897634$ of them, $28$ modulo the group $\Aut\Gamma_r$ of
order $2^{16} \cdot 3^2 \cdot 5 \cdot 7$:
\begin{enumerate}
\item odd:
$\wB_8 \wF_4^2$,
$\wE_7 \wC_5(2) \wF_4$,
$\wC_8(2) \wE_8$,
$\wD_8 \wC_4(2)^2$,
$\wB_6^2 \wC_4(2)$,
$\wE_6 \wA_5(2) \wB_5$,
$\wB_4^4$,
$\wB_{10} \wC_6(2)$,
$\wD_6 \wB_4 \wC_3(2)^2$,
$\wD_9 \wA_7(2)$,
$\wA_9 \wA_4(2) \wC_3(2)$,
$\wA_5^2 \wC_2(2) \wA_2(2)^2$,
$\wD_5^2 \wA_3(2)^2$,
$\wA_7 \wA_3(2) \wB_3^2$,
$\wD_4^2 \wC_2(2)^4$,
$\wA_3^4 \wA_1(2)^4$,
$\wA_1^{16}$.
\item even ord:
$\wE_7 \wC_6 \wB_3(2)$,
$\wD_{10} \wB_3(2)^2$,
$\wC_{12} \wF_4(2)$,
$\wE_8 \wF_4(2)^2$,
$\wD_8 \wC_4^2$,
$\wD_6^2 \wC_2^2$,
$\wC_8^2$,
$\wA_{11} \wC_3 \wA_2(2)$,
$\wE_6^2 \wA_2(2)^2$,
$\wA_7^2 \wA_1(2)^2$,
$\wD_4^4$.
\end{enumerate}

\subsection*{\boldmath$(22,2,0)$\unboldmath}\label{sec:22-2-0}

% This lattice is isomorphic to
% \begin{math}
%   U \oplus E_8^2 \oplus F_4.
% \end{math}
The Coxeter diagram
$\Gamma_r$ was described by
Borcherds \cite{borcherds1987automorphism-groups} and by Dolgachev-Kond\=o
\cite{dolgachev2003a-supersingular-k3}.  It consists of $42$ main
$(-2)$-roots corresponding to the $21$ points and $21$ lines of
$\bP^2(\bF_4)$, and $168$ additional $(-4)$-roots. 
$\Aut\Gamma_r = \PSL(3,\bF_4).D_{6}$ has order $241920$.  By Esselmann
\cite{esselmann1996uber-die-maximale} this lattice has the maximal
possible rank for a reflective hyperbolic lattice.

I computed the maximal parabolic subdiagrams of $\Gamma_r$. There are
$1095990$ of them, $18$ modulo $\Aut\Gamma_r$, all odd:
$\wE_7^2 \wB_6$, $\wE_8^2 \wF_4$, $\wA_7^2 \wD_5
\wA_1(2)$, $\wD_{12} \wB_8$, $\wB_{20}$, $\wB_{12} \wE_8$, $\wD_6^3
\wC_2(2)$, $\wD_{16} \wF_4$, $\wD_{10} \wE_7
\wC_3(2)$, $\wD_8^2 \wB_4$, $\wA_{15} \wB_5$, $\wA_9^2
\wC_2(2)$, $\wA_{11} \wE_6 \wB_3$, $\wA_{11} \wD_7 \wA_2(2)$,
$\wA_{17} \wC_3(2)$, $\wE_6^3 \wA_2(2)$, $\wA_5^4$, $\wD_4^5$.

\subsection*{\boldmath$(14,12,1)$\unboldmath}\label{sec:14-12-1}

$\Gamma_r$ has $17$ roots. There are $16$ maximal parabolic
subdiagrams, $9$ modulo $\Aut\Gamma_r=S_2$:
\begin{enumerate}
\item odd: $\wE_8(2) \wC_4(2)$, $\wC_{12}(2)$, $\wD_{12}(2)$, 
\item even ord: $\wD_8(2) \wB_4(2)$, $\wB_8(2) \wB_4(2)$, $\wB_8(2)
  \wF_4(2)$, $\wE_7(2) \wC_3 \wC_2$,
\item even char: $\wE_8(2) \wF_4(2)$, $\wB_{12}(2)$.
\end{enumerate}

\subsection*{\boldmath$(15,11,1)$\unboldmath}\label{sec:15-11-1}

% This lattice is isomorphic to
% \begin{displaymath}
%   U \oplus F_4 \oplus E_8(2) \oplus B_1(2) \simeq
%   U \oplus D_4 \oplus A_1^9.
% \end{displaymath}
There are $66$ roots which modulo
$\Aut\Gamma_r = S_2\ltimes(S_3\times S_3)$ split into five  orbits.
The subdiagram of the main roots is the graph $L(K_{3,3})$, the line
graph of $K_{3,3}$, a $4$-regular on $9$ vertices.
The other roots are additional.

The $(-4)$-roots are split into two groups of sizes $6$, and $9$,
connected to the main roots by $0$, resp.\ $2$ edges.  Together, they
form the graph $K_{3,3}^{(2)}$.  There are also two orbits of sizes
$6$ and $36$ connected to the main roots by $3$, resp.\ $5$ edges,
latter with multiplicities $2,2,2,4,4$.  These are $(-2)$-roots of
divisibility $2$.  By~\cite[Lem.~4.14]{alexeev2022mirror-symmetric}
removing one of these roots and its $(-2)$-neighbors gives the Coxeter
diagrams for the lattices $(14,10,0)$ and $(14,10,1)$.

There are $522$ maximal parabolic subdiagrams, $16$ modulo
$\Aut\Gamma_r=\Aut G$:
\begin{enumerate}
\item odd: $\wA_{11}(2)\wA_2$, $\wE_7(2)\wC_3(2)\wB_3$,
  $\wC_{12}(2)\wA_1$, $\wE_8(2)\wF_4\wA_1$, $\wC_8(2)\wC_5(2)$,
  $\wC_9(2)\wF_4$, $\wD_8(2)\wC_4(2)\wA_1$.
\item even ord: $\wE_6(2)\wA_5\wC_2$, $\wB_6(2)\wC_4\wB_3(2)$,
  $\wA_7(2)\wB_3(2)\wA_3$, $\wB_7(2)\wC_6$, $\wB_4^3(2)\wA_1$,
  $\wD_6(2)\wC_3^2\wA_1$, $\wB_8(2)\wF_4(2)\wA_1$,
  $\wE_7(2)\wC_5\wA_1$, $\wB_5(2)\wF_4^2(2)$.
\end{enumerate}

\subsection*{\boldmath$(15,13,1)$\unboldmath}\label{sec:15-13-1}

$\Gamma_r$ has $27$ roots. There are $46$ maximal parabolic
subdiagrams, $10$ modulo $\Aut\Gamma_r=S_3$:
\begin{enumerate}
\item odd: $\wE_8(2) \wC_5(2)$, $\wD_{12}(2) \wA_1$, $\wC_{13}(2)$,
\item even ord: $\wE_7(2) \wC_3 \wB_3(2)$, $\wB_8(2) \wB_5(2)$,
  $\wA_{11}(2) \wA_2$, $\wB_{12}(2) \wA_1$, $\wB_9(2) \wF_4(2)$,
  $\wD_8(2) \wB_4(2) \wA_1$, $\wE_8(2) \wF_4(2) \wA_1$. 
\end{enumerate}

\subsection*{\boldmath$(16,12,1)$\unboldmath}\label{sec:16-12-1}

% This lattice is isomorphic to
% \begin{math}
%   %U \oplus F_4 \oplus E_8(2) \oplus B_2(2) \simeq
%   U \oplus D_4 \oplus A_1^8. 
% \end{math}
There are $118$ roots which modulo $\Aut\Gamma_r = S_6$ split into
five orbits of sizes $15,15,6,10,72$ of vectors with square
$-2,-4,-4,-2,-4$. There are $15$ main roots forming $L(K_6)$, the line
graph of the complete graph $K_6$. If $L=(16,6,1)$ then $(16,12,1)$ is
the even sublattice of $L^*(2)$. This implies that the maximal
parabolic diagrams of the $(16,12,1)$ lattice are in a bijection with
those of the $(16,6,1)$ lattice, with the $(-2)$ and $(-4)$-roots
interchanged, and odd and even ordinary diagrams interchanged.

\subsection*{\boldmath$(16,14,1)$\unboldmath}\label{sec:16-14-1}

There are $29$ roots. This diagram is dual to that of $(16,4,1)$, with
the $(-2)$ and $(-4)$ roots switched.
There are $115$ maximal parabolic
subdiagrams, $14$ modulo $\Aut\Gamma_r=D_6$:

\begin{enumerate}
\item odd: $\wE_8(2) \wC_6(2)$, $\wC_{14}(2)$, $\wD_{12}(2) \wC_2(2)$,
  $\wE_7(2) \wE_7(2)$, 
\item even ord: $\wE_8(2) \wF_4(2) \wC_2$, $\wD_8(2) \wB_4(2) \wC_2$,
  $\wB_{12}(2) \wC_2$, $\wD_{10}(2) \wC_3 \wA_1(2)$,
\newline
  $\wE_7(2)
  \wB_4(2) \wC_3$, $\wE_7(2) \wB_6(2) \wA_1(2)$, $\wB_8(2) \wB_6(2)$,
  $\wB_{10}(2) \wF_4(2)$, $\wA_{11}(2) \wA_2 \wA_1$, $\wD_6(2) \wD_6(2) \wC_2$.
\end{enumerate}

% There are $8917$ maximal parabolic subdiagrams, $28$ modulo
% $\Aut\Gamma_r = \Aut G$:
% \begin{enumerate}
% \item odd: $\wD_6^2(2)\wC_2(2)$, $\wD_8(2)\wC_4(2)\wC_2(2)$,
%   $\wA_{11}(2)\wA_2\wA_1$, $\wC_8(2)\wC_6(2)$, $\wC_{12}(2)\wC_2(2)$,
%   \linebreak
%   $\wE_7(2)\wC_4(2)\wB_3$, $\wE_7(2)\wC_6(2)\wA_1(2)$,
%   $\wE_8(2)\wF_4\wC_2(2)$, $\wC_{10}(2)\wF_4$,
%   $\wD_{10}(2)\wB_3\wA_1(2)$, 
  
% \item even ordinary: $\wD_4^2(2)\wC_2^3$, $\wB_4^3(2)\wC_2$,
%   $\wE_6(2)\wA_5\wB_3(2)$, $\wA_9(2)\wA_4\wA_1(2)$,
%   \linebreak
%   $\wD_6(2)\wC_3^2\wC_2$, $\wB_6(2)\wF_4^2(2)$, $\wB_8(2)\wC_6$,
%   $\wB_8(2)\wF_4(2)\wC_2$, $\wD_6(2)\wB_4(2)\wC_3\wA_1(2)$,
%   \linebreak
%   $\wA_7(2)\wA_3\wB_3(2)\wA_1$, $\wE_8(2)\wC_6$, $\wE_7(2)\wC_5\wC_2$,
%   $\wA_5^2(2)\wA_2^2$, $\wD_8(2)\wC_4\wC_2$, $\wE_7(2)\wF_4(2)\wC_3$,
%   $\wB_{10}(2)\wC_4$, $\wB_6^2(2)\wC_2$, $\wB_6(2)\wB_4(2)\wC_4$.
% \end{enumerate}

% \begin{remark}
%   By~\cite[Lem.~5.9]{alexeev2022mirror-symmetric}, one consequence of
%   listing maximal parabolic subdiagrams in a hyperbolic diagram
%   $(r,a,1)$ is the classification of the negative definite even
%   $2$-elementary lattices with the invariants $(r-2,a,1)$,
%   $(r-2,a-2,1)$ and $(r-2,a-2,0)$.
% \end{remark}

\subsection*{Other coeven lattices}

If $S=(r,a,0)$ is an even coeven lattice then so is $S^\dag = S^*(2) =
(r,r-a,0)$. Their Coxeter diagrams are dual, with $(-2)$ and
$(-4)$-roots interchanged, and there is a bijection between the
maximal parabolic subdiagrams. The remaining lattices in
Fig.~\ref{fig:reflective} are obtained from the ones already listed
here or in \cite{alexeev2022mirror-symmetric}. This applies to the
following lattices:
\begin{displaymath}
  (14,10,0), (14,12,0), (18,12,0), (18,14,0), (18,16,0), (22,20,0).
\end{displaymath}

\section{Applications to infinite automorphism groups of K3 surfaces}
\label{sec:automorphisms}

The main application of the Coxeter diagrams
in~\cite{vinberg1983two-most} was the description, for the very first
time, of two infinite automorphism groups of K3 surfaces. In the same
way our computations imply such a description in several more cases.

A K3 surface $X$ whose Picard lattice $S_X$ is $2$-elementary comes
with a canonical nonsymplectic involution $\iota$ on $X$, acting as
$+1$ on $S$ and as $-1$ on the lattice of transcendental cycles
$T=S^\perp$ in $\lk$. Any automorphism of $X$ commutes with
$\iota$. Let $Y=X/\iota$ be the quotient surface. It is a rational
surface unless $S=(10,0,0)$ in which case $Y$ is an Enriques
surface. Let $B$ be the branch divisor of $\pi\colon X\to Y$ and
denote by $\Aut'(Y,B)$ the subgroup of $\Aut(Y,B)$ acting trivially on
$\Pic Y$.

\begin{lemma}
  Let $S$ be a $2$-elementary K3 lattice $S\ne (10,10,0)$, 
  $\Gamma_r$ the Coxeter diagram of $S$, and let $X$ be a K3
  surface with Picard lattice $S_X=S$. Then there exists an exact
  sequence
  \begin{displaymath}
    0\to \la\iota\ra \times \Aut'(Y,B)
    \to \Aut X \to \Sym\Gamma_r \ltimes W(\Gamma_4) \to 0,
  \end{displaymath}
  where $\Gamma_4\subset\Gamma_r$ is the subdiagram formed by the
  $(-4)$-roots, and $W(\Gamma_4)$ is the corresponding Coxeter group.
\end{lemma}
\begin{proof}
  A well-known application of Torelli theorem
  \cite{piateski-shapiro1971torelli} says that for a projective K3
  surface with an ample cone $A(X)$ the natural homomorphism 
  \begin{displaymath}
    \rho\colon \Aut X \to \Aut A(X) = O^+(S) / W_2(S) = O(S) / \pm W_2(S),
  \end{displaymath}
  has finite kernel and cokernel. The kernel of $\rho$ consists of
  automorphisms that act trivially on $S_X$. They descend to
  automorphisms of the pair $(Y,B)$, and $\iota$ is the only one
  descending to the identity. Vice versa, an automorphism of $(Y,B)$
  lifts to an automorphism of $X$ since $X = \Spec \cO\oplus L\inv$
  with $L^2\simeq \cO(B)$ and $\Pic Y$ has no $2$-torsion.
  $(\Pic Y)\otimes\bQ$ is identified with
  $(\Pic X)^\iota \otimes\bQ=S\otimes\bQ$.  So, the automorphisms of
  $X$ acting trivially on $S$ descend to $\Aut'(Y,D)$. Thus,
  $\ker\rho = \bZ_2\times \Aut(Y,B)$.

  The homomorphism $\rho$ is surjective.  Indeed, an isometry of
  $\lk$ is equivalent to a pair of isometries of $S$ and $T$ which
  have the same image in the finite isometry group of the discriminant
  group under 
  $O(S)\to O(A_S,q_S) = O(A_T,q_T) \gets
  O(T)$. By~\cite{nikulin1979integer-symmetric} for an indefinite
  $2$-elementary lattice $H$, such as $S$ and $T$, the homomorphism
  $O(H)\to O(A_H,q_H)$ is surjective. So any element of $O(S)$ can be
  lifted to $a\in O(\lk)$. Then a composition $w\circ a$ with some
  $w\in W_2(S)$ sends the fundamental domain $A(X)$ to itself and
  therefore is induced by some $g\in \Aut X$.
 
  By~\cite[Prop., p.2]{vinberg1983two-most}
  or~\cite[Prop.~2.4]{alexeev2006del-pezzo} the quotient group
  $W(\Gamma_r)/W_2(S)$ is isomorphic to $W(\Gamma_4)$. Together
  with the equality $O^+(S) = \Sym\Gamma_r\ltimes W(\Gamma_r)$ this
  gives $O^+(S)/W_2(S) = \Sym\Gamma_r \ltimes W(\Gamma_4)$.
\end{proof}

By~\cite[Sec.~3]{alexeev2022mirror-symmetric}, for the $2$-elementary
K3 lattices on the $r+a=20$ line the groups
$\Sym\Gamma_r\ltimes W(\Gamma_4)$ are extensions of affine groups
$W(\wE_n)$ by dihedral groups.
The diagrams of this paper provide even
more interesting examples.
The $(-4)$-subdiagrams for the K3 lattices with $r+a=22$ were given in
the previous section. We list them in
Table~\ref{tab:automorphisms}. $T_{p,q,r}$ denotes a tree with
legs of lengths $p$, $q$, $r$.

\begin{table}[]
  \centering
  \begin{tabular}{lll}
    Lattice& $\Aut\Gamma_r$&$\Gamma_4$\\
    \toprule
    $(11,11,1)$ & $1$ &  $T_{2,3,7}$\\
    $(12,10,1)$ & $1$ &  $T_{2,4,6} \sqcup A_1$\\
    $(13,9,1)$ & $S_3$ & $T_{4,4,4}$ \\
    $(14,8,1)$ & $S_4$ & a $10$-vertex trivalent tree with $6$ ends\\
    $(15,7,1)$ & $S_5$ & the Petersen graph \\
    $(16,6,1)$ & $S_6$ & a diagram with $10+15$ vertices \\
    $(18,4,0)$ & $S_2\ltimes(S_4\times S_4)$ & a diagram with $24$ vertices\\
    $(20,2,1)$ & $S_5$ & $K_5$ with bold edges \\
    
  \end{tabular}
  \smallskip
  \caption[$(-4)$-subgraphs $\Gamma_4\subset\Gamma_r$]{Coxeter
    diagrams $\Gamma_4$ defining $\Aut X$}\label{tab:automorphisms}
\end{table}

The diagrams $\Gamma_4$ for $(16,6,1)$ and $(18,4,0)$ are quite
complicated, and have both simple and bold edges.
It is a little easier to describe the subdiagram $\Gamma'_4$ of
$\Gamma_4$ with the bold edges removed. The Coxeter group
$W(\Gamma_4')$ is the quotient of $W(\Gamma_4)$ by the additional
commuting relations for the bold edges.

For $(16,6,1)$ the graph $\Gamma_4'$ is a union of $10$ disjoint
vertices and $L(K_6)$.

For $(18,4,0)$ the graph $\Gamma_4'$ is a union of two disjoint
subgraphs $G_1\sqcup G_2$, and each $G_i$ can be defined as the
complement of $K_4\sqcup K_4\sqcup K_4$ in the complete graph
$K_{12}$: an edge of $K_{12}$ is in $G_i$ iff it is
not in $K_4\sqcup K_4\sqcup K_4$. The two $12$-element sets can be
identified with the cosets of $A_4$ in $S_4$ and the three $4$-element
sets with the cosets of the Klein four-group $V_4$ in $A_4$.

\bibliographystyle{amsalpha}
%\bibliography{va}

\def\cprime{$'$}
\providecommand{\bysame}{\leavevmode\hbox to3em{\hrulefill}\thinspace}
\providecommand{\MR}{\relax\ifhmode\unskip\space\fi MR }
\providecommand{\MRhref}[2]{%
  \href{http://www.ams.org/mathscinet-getitem?mr=#1}{#2}
}
\providecommand{\href}[2]{#2}

\end{document}